\documentclass[12pt]{amsart}
\usepackage[top=1in, bottom=1in, left=1in, right=1in]{geometry}

\usepackage{amssymb,mathrsfs,amsmath}

\newtheorem{theorem}{Theorem}[section]
\newtheorem{conjecture}{Conjecture}[section]

\newtheorem{corollary}[theorem]{Corollary}
\theoremstyle{remark}

\newcommand{\bsp}{\begin{split}}
\newcommand{\esp}{\end{split}}
\newcommand{\be}{\begin{equation}}
\newcommand{\ee}{\end{equation}}
\newcommand{\bes}{\begin{equation*}}
\newcommand{\ees}{\end{equation*}}
\newcommand{\bv}\boldsymbol{}

\numberwithin{equation}{section}

\renewcommand{\mod}[1]{~{\rm mod}\,#1}

\begin{document}
\title{A Lower Bound For Biases Amongst Products Of Two Primes}
\author{Patrick Hough}
\address{Department of Mathematics\\
University College London\\
Gower Street\\
London, WC1E  6BT\\
England}
\email{{\tt patrick.hough.12@ucl.ac.uk}}
\keywords{Primes in arithmetic progressions, Prime races}
\thanks{This paper comes as a result of my MSc dissertation at University College London. I give greatest thanks to my supervisor Prof. Andrew Granville. I also thank Ronnie George and Tarquin Grossman for their general exuberance.}
\date{\today}

\begin{abstract} We establish a conjectured result of Dummit, Granville and Kisilevsky for the maximum bias in P$2$ races, which compare the number of P$2's\leq x$ whose prime factors, evaluated at a given quadratic character, take specific values. 
\end{abstract}

\maketitle


\section{Introduction} 
In $2015$, Dummit, Granville and Kisilevsky~\cite{DGK} showed that significantly more than a quarter of all odd integers of the form $pq$ up to $x$,  with $p$ and $q$ both prime, satisfy $p\equiv q\equiv 3\mod{4}$. More generally they proved the following Theorem. 
\begin{theorem}[Dummit, Granville and Kisilevsky, 2016]\label{DGK} Let $\chi$ be a quadratic character of conductor $d$. For $\eta=-1$ or $1$ we have \be\frac{\#\{pq\leq x:\:\chi_d(p)=\chi_d(q)=\eta\}}{\frac{1}{4}\#\{pq\leq x: \: (pq,d)=1\}}=1+\eta\frac{(\mathcal{L}_{\chi_d}+o(1))}{\log{\log{x}}},\:\:\:\text{ where } \mathcal{L}_{\chi_d}:=\sum_{p}\frac{\chi_d(p)}{p}.\ee
\end{theorem}
We focus on just how large this bias can get if we restrict the conductor of our quadratic character by $d\leq x$. The following prediction was made in relation to this problem.
\begin{conjecture}[Dummit, Granville and Kisilevsky, 2016]
There exists $d\leq x$ such that 
\[\label{eq:DGK}\frac{\#\{pq\leq x:\:\chi_d(p)=\chi_d(q)=\eta\}}{\frac{1}{4}\#\{pq\leq x: \: (pq,d)=1\}} \text{ is as large as }1+\frac{\log\log\log{x}+O(1)}{\log\log{x}}.\] 
\end{conjecture}\label{DGKCon}
It is important to note that in proving such a claim, one must prove a uniform version of Theorem \ref{DGK} since the proof there assumes that $x$ is allowed to be very large compared to $d$. Whilst this is not a straightforward task, once achieved, all that remains is to show that $\mathcal{L}_{\chi_d}$ can be found suitably large for $d\leq x$. It turns out that this final step is easily accomplished using a slight adaptation of a result from Granville and Soundararajan's work on extremal values of $L(1,\chi)$~\cite{GS}. \\\\ In this paper, we prove Conjecture \ref{DGKCon} and indeed give a slightly stronger result in the form of the following Theorem.
\begin{theorem}\label{hough1}
Fix $\epsilon>0$. For large $x$ there exist at least $D(x)^{1-\epsilon}$ integers $d\leq D(x)$ for which 
\[\frac{\#\{pq\leq x:\:\chi_{d}(p)=\chi_{d}(q)=\eta\}}{\frac{1}{4}\#\{pq\leq x:\:(pq,{d})=1\}}\:\:\:\text{ is at least as large as }\:\:\:1+\frac{\log\log\log{x}+O(1)}{\log\log{x}}.\] 
Here, $D(x)=\exp[C(\log{x})^{1/2}]$.
\end{theorem}
In $1919$, Littlewood gave an upper bound on the value of $L(1,\chi)$ as follows. 
\begin{theorem}[Littlewood, 1919]\label{Littlewood}
Assume the Generalised Riemann Hypothesis (GRH). For any non-principal primitive character $\chi$ (mod $q$), one has 
\[|L(1,\chi)|\leq (2e^{\gamma}+o(1))\log\log{q}.\]
\end{theorem}
This upper bound, when combined with a lower bound derived from~\cite{GS}, allows us to claim that, under GRH, we have found the maximum bias amongst these races. 
\begin{theorem}\label{thm:Hough2}
\label{max} Assume GRH. Then
\[\max_{d\leq x}\left|\frac{\#\{pq\leq x:\:\chi_{d}(p)=\chi_{d}(q)=\eta\}}{\frac{1}{4}\#\{pq\leq x:\:(pq,{d})=1\}}-1\right|\sim\frac{\log\log\log{x}}{\log\log{x}}.\]
\end{theorem}\vspace{+2ex}
Using a second derived result from~\cite{GS}, the initial legwork of proving a uniform version of Theorem \ref{DGK} allows us to show just how \emph{small} this bias can get. 
\begin{theorem}\label{thm:Hough3}
Fix $\epsilon>0$. For large $x$ there exist at least $D(x)^{1-\epsilon}$ integers $d\leq D(x)$ for which 
\[\frac{\#\{pq\leq x:\:\chi_{d}(p)=\chi_{d}(q)=\eta\}}{\frac{1}{4}\#\{pq\leq x:\:(pq,{d})=1\}}\:\:\:\text{ is at least as small as }\:\:\:1-\frac{\log\log\log{x}+O(1)}{\log\log{x}}.\] Here, $D(x)=\exp[C(\log{x})^{1/2}]$.
\end{theorem}\vspace{+2ex}
Finally, we conducted a computational search for conductors $d$ that give rise to particularly biased prime races of this kind. In order to find these large biases, is it crucial to study the quantity $\mathcal{L}_{\chi_d}$. Indeed, if $\chi_d(p)=-1$ for a large proportion of the small primes $p$ or alternatively $\chi_d(p)=+1$, then we expect to have a prime race with a correspondingly large bias. Such characters are closely linked to prime generating polynomials and it was the work of Jacobson Jr. and Williams that gave us the best results in this area~\cite{JW}. We are thus able to give, what we believe to be, the lowest known value of $\mathcal{L}_{\chi_d}$ where the conductor $d$ of our character is \[133007243922787512412600341028518035429251391005992761399935498154029253, \]
and $\mathcal{L}_{(d/\cdot)}\approx -2.1108$. In the words of Fiorilli and Martin~\cite{FM}, this conductor gives rise to the `most unfair' known such prime race. Moreover, the corresponding value of $L(1,\chi_d)=0.144$, which is so important in determining $\mathcal{L}_{\chi_d}$, is the lowest calculated for a real Dirichlet character $\chi$. Indeed, extremely small values of $L(1,\chi)$ have many links with prime generating polynomials. 

\section{Proof Of Theorem \ref{hough1}} 
We begin by defining the notion of a Siegel zero for a Dirichlet $L$-function $L(s,\chi)$ associated with the real Dirichlet character $\chi$ of modulus $q$. A zero of $L(s,\chi)$ is called a Siegel zero if, for some suitable positive constant $c$, $L(s,\chi)$ has a real zero $\beta$ such that \[1-\frac{c}{\log{q}}\leq \beta\leq 1.\]

As mentioned, a proof of Theorem \ref{hough1} requires that we give a uniform version of Theorem \ref{DGK}. Moreover, we need that \eqref{eq:DGK} holds for $d\leq D(x):=\exp[C(\log{x})^{1/2}]$. We focus on the case in which $\eta=1$ since the case $\eta=-1$ is tackled using much the same reasoning. Establishing this result, assuming $L(s,\chi_d)$ has no Siegel zero, will make up the majority of the proof. To complete the proof we show that, for large $x$, there exist many $d\leq D(x)$ corresponding to real Dirichlet characters, such that \be\label{Lchi} \mathcal{L}_{\chi_d}\geq \log\log\log{x}+O(1), \ee where $L(s,\chi_d)$ has no Siegel zero. The following application of Theorem $1$ from~\cite{GS} is sufficient in order to make this last step. Throughout this paper we will denote the $j$-fold iterated logarithm by $\log_j$ where appropriate.
\begin{theorem}[Application of Theorem $1$ From~\cite{GS}]\label{mod1}
Let $D$ be large and $\tau=\log_2(D)-2\log_3(D)$. Then \[\Phi_D(\tau)\gg D^{-\epsilon}, \text{ for fixed } \epsilon>0,\] where $\Phi_D(\tau)$ is the proportion of fundamental discriminants with $d\leq D$ for which
\[ L(1,\chi_d)\geq e^{\gamma}\tau.\]
\end{theorem}
\begin{proof}
The proof of this result is simply computational. Fix $\epsilon>0$.We take $\tau=\log_2(D)-2\log_3(D)$ in Theorem $1$ of \cite{GS}. By Theorem $1$ and then Proposition $1$ of~\cite{GS} we have (where $C_1$ is given explicitly in Proposition $1$ of~\cite{GS}):
\begin{align*}\Phi_D(\log_2(D) - 2 \log_3(D)) &= \exp\left[-\frac{ e^{\log_2(D) - 2\log_3(D)-C_1}}{(\log_2{D} - 2 \log_3(D))} \left(1+O\left(\frac{1}{\log_2(D)}\right)\right)\right]\left( 1+O\left(\frac{1}{\log_3(D)}\right) \right).
\end{align*}
Now we note that \[ \frac{1}{(\log_2(D)-2\log_3(D))}\left(1+O\left(\frac{1}{\log_2(D)}\right)\right)>\frac{1}{\log_2(D)}, \text{ for } D \text{ sufficiently large}.\] Therefore, 
\begin{align*} \Phi_D(\tau) &>\exp\left(-\frac{e^{-C_1}\log(D)}{(\log_2(D))^3}\right), \\
&> D^{-\epsilon}, \text{ for } D \text{ sufficiently large}. \end{align*}
\end{proof}
Theorem \ref{mod1} ensures that there are, for sufficiently large $D$, at least $D^{1-\epsilon}$ fundamental discriminants $d\leq D$ for which $L(1,\chi_d)\geq e^{\gamma}(\log_2(D)-2\log_3(D))$. This follows from the well established fact that there are $\frac{6}{\pi^2}x+O(x^{\frac{1}{2}+\epsilon})$ fundamental discriminants $d$ with $|d|\leq x$. 
Note that we may write 
\be\label{curlyL} \mathcal{L}_{\chi}=\sum_{p}\frac{\chi(p)}{p}=\sum_{m\geq 1}\frac{\mu(m)}{m}\log{L(m,\chi^m)}=\log{L(1,\chi)}+E(\chi), \ee where \[\sum_{p}\left(\log\left(1-\frac{1}{p}\right)+\frac{1}{p}\right)=-0.315718...\leq E(\chi)\leq \sum_{p}\left(\log\left(1+\frac{1}{p}\right)-\frac{1}{p}\right)=-0.18198...\] 
Therefore for sufficiently large $D$, there are at least $D^{1-\epsilon}$ fundamental discriminants $d\leq D$ for which  $\mathcal{L}_{\chi_d} \geq \log_3(D) +O(1)$. 

In light of a result by Page that ensures there is never more than one conductor $d\leq x$ for which $L(s,\chi_d)$ has a Siegel zero, Theorem \ref{mod1} combines with \eqref{curlyL} to give \eqref{Lchi} as required. 
\begin{corollary}\label{cor:mod1}
Fix $\epsilon>0$. For sufficiently large $D$, there are at least $D^{1-\epsilon}$ fundamental discriminants $d\leq D$ for which  \[\mathcal{L}_{\chi_d} \geq \log_3(D) +O(1),\] 
such that $L(s,\chi_d)$ does not have a Siegel zero.
\end{corollary}

\subsection{Proof of uniform version of \eqref{DGK} with $L(s,\chi)$ having no Siegel zero}\leavevmode
\vspace{+1ex}
We write 
\begin{align}\label{original} \frac{\#\{ab\leq x:\:\chi(a)=\chi(b)=1\}}{\frac{1}{4}\#\{ab\leq x:\:(ab,d)=1\}}&=\sum_{\substack{ab\leq x \\ a,b\text{ prime}}}(1+\chi(a))(1+\chi(b))\Bigg{/}\sum_{\substack{ab\leq x\\ a,b\text{ prime}}}1, \nonumber\\
&=1+\sum_{\substack{ab\leq x \\ a,b\text{ prime}}}\big(\chi(a)+\chi(b)+\chi(ab)\big)\Bigg{/}\sum_{\substack{ab\leq x\\ a,b\text{ prime}}}1, \nonumber \\
&=1+\frac{\log{x}}{2x(\log\log{x}+O(1))}\sum_{\substack{ab\leq x \\ a,b\text{ prime}}}(\chi(a)+\chi(b)+\chi(ab)),
\end{align}
where we have used the following result due to Landau
\[\sum_{\substack{ab\leq x\\ a,b\text{ prime}}}1=\frac{2x}{\log{x}}\left(\log\log{x}+O(1)\right).\]
Here we count each of the products $ab$ and $ba$ in the sum separately, which accounts for the factor of $2$ used here in contrast with the the usual formulation of this result. 
Let us consider the first sum here. 
\begin{equation}\label{examine} \sum_{\substack{ab\leq x \\ a,b\text{ prime}}}\chi(a)= 
\sum_{\substack{a\leq \sqrt{x}\\ a\text{ prime}}}\chi(a)\sum_{\substack{b\leq x/a\\ b \text{ prime}}}1+\sum_{\substack{b\leq \sqrt{x}\\ b \text{ prime}}}1\sum_{\substack{a\leq x/b\\ a\text{ prime}}}\chi(a)-\sum_{\substack{a\leq \sqrt{x}\\a\text{ prime}}}\chi(a)\cdot\sum_{\substack{b\leq \sqrt{x}\\b \text{ prime}}}1.\end{equation}
To get a useful bound for the terms in this expression it is clearly necessary to understand the quantity \be\label{quant}\sum_{\substack{n\leq x\\n\text{ prime}}}\chi(n).\ee
We now incorporate the notion of a Siegel zero by way of a Theorem from~\cite{DAV}, pg. $95$.
\begin{theorem}[Page, $1935$]\label{thm:page}
If $c$ is a suitable positive constant, there is at most one real primitive character $\chi$ to a modulus $d\leq x$ for which $L(s,\chi)$ has a real zero $\beta$ satisfying 
\be\label{eq:page}\beta>1-\frac{c}{\log{x}}.\ee Such a zero is called a Siegel zero.
\end{theorem} 
Davenport also provides us with the next step. In Chapter $14$ it is shown that, restricting the conductor of our character $\chi_d$ by $d\leq \exp[C(\log{x})^{\frac{1}{2}}]$ and assuming that the corresponding $L$-function has no zero lying in the region defined by \eqref{eq:page}, we have 
\be\label{eq:psi} \psi(x,\chi)\ll \frac{x}{e^{c(\log{x})^\frac{1}{2}}}, \ee where $c$ is some positive constant. Since we are only interested in summing over the primes and not prime powers, we need the following deduction from \eqref{eq:psi} \be\label{eq:theta} \theta(x,\chi)\ll \frac{x}{e^{c(\log{x})^{\frac{1}{2}}}}\ll\frac{x}{(\log{x})^{N}},\ee for any given $N>0$ where the implicit constant depends only on $N$. Here we have used that the contribution of primes powers in \eqref{eq:psi} is $\ll x^{1/2}$. This quantity turns out to be crucial since we will need to bound the sum of $\chi_d(p)$ over the primes. We proceed using partial summation but must now be careful as to how we bound $\theta(t,\chi)$. For $d<\exp[c(\log{t})^{1/2}]$ or equivalently $t>\exp[(\frac{1}{c}\log{d})^2]$, we use \eqref{eq:theta}. 
If $t<\exp[(\frac{1}{c}\log{d})^{2}]=:x_d$ then we use the trivial bound \[ \theta(t,\chi)\ll t.\] With this consideration and on using partial summation, \eqref{quant} becomes 
\begin{align}\label{long}
\sum_{\substack{n\leq x\\ n\text{ prime}}}\chi(n)&\ll\sum_{\substack{n\leq x_d\\n\text{ prime}}}\chi(n)+\int_{x_d}^{x}\frac{1}{\log{t}}\,d\theta(t,\chi),\nonumber \\
&\ll x_d+ \left[\frac{t}{e^{c(\log{t})^{\frac{1}{2}}}\log{t}}\right]^{x}_{x_d}+\int_{x_d}^{x}\frac{1}{(\log{t})^{N+2}}\, dt, \nonumber \\
&\ll e^{(\frac{1}{c}\log{d})^2}+\frac{x}{(\log{x})^{N+1}},
\end{align}
for $N>0$. Here we have used the fact that $|\chi(n)|=1$, \eqref{eq:theta}, and the relation \be Li_N(x)=\frac{x}{(\log{x})^N}+O\left(\frac{x}{(\log{x})^{N+1}}\right). \ee
Now, applying the constraint \[ d\leq \exp[c'(\log{x})^{1/2}],\] where we choose $0<c'<c$, the first term in \eqref{long} can be bounded as follows. 
 \begin{align*} e^{(\frac{1}{c}\log{d})^2}&\ll \frac{x}{(\log{x})^{N}}\text{ for } N>0.\end{align*}
 The result of this is that 
 \be\label{nobadres} \sum_{\substack{n\leq x\\ n\text{ prime}}}\chi(n)
\ll \frac{x}{(\log{x})^{N}},  \ee for any given $N>0$, where again the implicit constant depends only on $N$.
Using \eqref{nobadres}, we may bound the second term of \eqref{examine} as follows. 
\[\sum_{\substack{b\leq\sqrt{x}\\b\text{ prime}}}1\sum_{\substack{a\leq x/b\\ a \text{ prime}}}\chi(a)\ll \frac{x}{(\log{x})^{N+1}}\sum_{\substack{b\leq\sqrt{x}\\b\text{ prime}}}\frac{1}{b}\ll \frac{x}{(\log{x})^{N}}, \text{ where }N>1.\]
Using The Prime Number Theorem, the third term of \eqref{examine} can be similarly bounded. \[ \sum_{\substack{a\leq \sqrt{x}\\a\text{ prime}}}\chi(a)\cdot\sum_{\substack{b\leq \sqrt{x}\\b \text{ prime}}}1 \ll \frac{x^{1/2}}{(\log{x})^{N}}\cdot \frac{x^{1/2}}{\log{x}}\ll \frac{x}{(\log{x})^{N+1}}.\] 
The first term may be written as 
\begin{align}\label{last1} 
\sum_{\substack{a\leq \sqrt{x}\\ a\text{ prime}}}\chi(a)\sum_{\substack{b\leq x/a\\ b \text{ prime}}}1&
&=\sum_{\substack{a\leq\sqrt{x}\\a \text{ prime}}}\chi(a)Li(x/a)+O\left(\frac{x}{(\log{x})^{N}}\right), \:\:\text{ where we take }N>0. \end{align} 
Here we have used the following result, proved by La Vallée Poussin in 1899. \[\pi(x)=Li(x)+O\left(e^{-a(\log{x})^{1/2}}\right), \] for some positive constant $a$. These techniques, by symmetry, yield the same result for the second sum of \eqref{original}. Summarising, we have shown that 
\[\sum_{\substack{ab\leq x\\a,b \text{ prime}}}(\chi(a)+\chi(b))=2\sum_{\substack{a\leq\sqrt{x}\\a \text{ prime}}}\chi(a)Li(x/a)+O\left(\frac{x}{(\log{x})^{N}}\right), \:\:\text{ where we take }N>0.\]
We now deal with the remaining term in \eqref{original}. Namely 
\[\sum_{\substack{ab\leq x\\ a,b\text{ prime}}}\chi(ab)=\sum_{\substack{a\leq \sqrt{x}\\a\text{ prime}}}\chi(a)\sum_{\substack{b\leq x/a\\a \text{ prime}}}\chi(b)+\sum_{\substack{b\leq \sqrt{x}\\b\text{ prime}}}\chi(b)\sum_{\substack{a\leq x/b\\a\text{ prime}}}\chi(a)-\sum_{\substack{a\leq\sqrt{x}\\a \text{ prime}}}\chi(a)\cdot\sum_{\substack{b\leq \sqrt{x}\\b \text{ prime}}}\chi(b).\]
Applying \eqref{nobadres} to each the three terms allows us to bound them by $x/(\log{x})^N$ with $N>1$. Indeed, the first term becomes \[\sum_{\substack{a\leq \sqrt{x}\\a\text{ prime}}}\chi(a)\sum_{\substack{b\leq x/a\\a \text{ prime}}}\chi(b)=\frac{x}{(\log{x})^{N+1}}\sum_{\substack{a\leq\sqrt{x}\\a\text{ prime}}}\frac{1}{a}\ll\frac{x}{(\log{x})^{N}},\text{ where }N>0.\]
Simply interchanging $a$ and $b$ yields the same bound for the second term and the third term is dealt with simply by taking the product of the two sums after applying \eqref{nobadres}. \\\\
Thus, we have for $N>0$, \begin{align*} \frac{\#\{ab\leq x:\:\chi(a)=\chi(b)=1\}}{\frac{1}{4}\#\{ab\leq x:\:(ab,d)=1\}} = 1+\frac{\log{x}}{x(\log\log{x}+O(1))}&\sum_{\substack{a\leq\sqrt{x}\\a\text{ prime}}}\chi(a)Li(x/a)\\&+O\left(\frac{1}{(\log\log{x}+O(1))(\log{x})^{N}}\right). \end{align*}
Now let us examine 
\be\label{pen}\sum_{\substack{a\leq\sqrt{x}\\a\text{ prime}}}\chi(a)Li(x/a) = \sum_{\substack{a\leq\sqrt{x}\\a\text{ prime}}}\frac{x\chi(a)}{a\log(x/a)}+O\left(x\sum_{\substack{a\leq\sqrt{x}\\a\text{ prime}}}\frac{\chi(a)}{a(\log(x/a))^{2}}\right). \ee
The last term can be written
\begin{align*}\label{sameagain} \left|x\sum_{\substack{a\leq \sqrt{x}\\a\text{ prime}}}\frac{\chi(a)}{a(\log(x/a))^2}\right|
&\ll \frac{x}{(\log{x})^2}\sum_{\substack{a\leq \sqrt{x}\\a\text{ prime}}}\frac{1}{a}\:\:\:\text{ using } |\chi(a)|\leq 1,\nonumber \\
&\ll \frac{x\log\log{x}}{(\log{x})^2}, \nonumber\\ \end{align*}
where we have used the following result 
\[ \sum_{\substack{a\leq{x}\\a\text{ prime}}}\frac{1}{a}=\log\log{x}+O(1).\]
We may also change the summand of the first term, only incurring a small error term so that it becomes \[\sum_{\substack{a\leq\sqrt{x}\\a\text{ prime}}}\frac{x\chi(a)}{a\log{x}},\]
where we have used the result
\[ \sum_{\substack{p\leq x \\ p\text{ prime}}}\frac{\log{p}}{p}=\log{x}+O(1).\]
The difference between this sum and the original can be bounded by \[\sum_{\substack{a\leq \sqrt{x} \\ a\text{ prime}}}\frac{x\log{a}}{a(\log{x})^2}\ll \frac{x}{\log{x}},\] which represents the error incurred in changing the summand of the first term in \eqref{pen}.
The result of this is that
\be\label{eq:pen}\sum_{\substack{a\leq\sqrt{x}\\a\text{ prime}}}\chi(a)Li(x/a)  =\frac{x}{\log{x}}\sum_{a \text{ prime}}\frac{\chi(a)}{a}-\frac{x}{\log{x}}\sum_{\substack{a>\sqrt{x}\\a\text{ prime}}}\frac{\chi(a)}{a}+O\left(\frac{x}{\log{x}}\right).\ee
Applying \eqref{nobadres} to the second term here, \eqref{pen} becomes 
\[ \sum_{\substack{a\leq\sqrt{x}\\a\text{ prime}}}\chi(a)Li(x/a) = \frac{x}{\log{x}}\left(\mathcal{L}_{\chi}+O(1)\right).\] 
The above calculations imply that, for large $x$ and $d\leq \exp[C(\log{x})^{1/2}]$ where $L(s,\chi_d)$ has no `exceptional' zero, 
\begin{equation}\label{step0}  \frac{\#\{ab\leq x:\:\chi(a)=\chi(b)=1\}}{\frac{1}{4}\#\{ab\leq x:\:(ab,d)=1\}} = 1+\frac{\mathcal{L_\chi}+O(1)}{\log\log{x}(1+O(1/\log\log{x})},\end{equation}
as required. 
Now, denoting $b=-O(1/\log\log{x})$, we have 
\[ \frac{1}{1+O(1/\log\log{x})}=\frac{1}{1-b}=1+b+b^2+...\] For large $x$, $b<1/2$ and \[1+b+b^2+...<1+2b=1+O(b).\] Thus, \[\frac{1}{1+O(1/\log\log{x})}=1+O(1/\log\log{x}),\] and \eqref{step0} becomes 
\begin{align}\label{step}  \frac{\#\{ab\leq x:\:\chi(a)=\chi(b)=1\}}{\frac{1}{4}\#\{ab\leq x:\:(ab,d)=1\}} &= 1+\frac{\mathcal{L_\chi}+O(1)+O(\mathcal{L}_{\chi}/\log\log{x})}{\log\log{x}}, \nonumber \\ &= 1+\frac{\mathcal{L_\chi}+O(1)}{\log\log{x}},
\end{align} as required, where we have used \eqref{curlyL} and that $L(1,\chi)<c\log{x}$ for some positive constant $c$.

\begin{proof}(of Theorem \ref{hough1})
Defining $D(x)=\exp[C(\log{x})^{1/2}]$, Corollary \ref{cor:mod1} gives us that there are at least $D(x)^{1-\epsilon}$ integers $d\leq D(x)$ such that the corresponding Dirichlet character $\chi_d$ is real and \[\mathcal{L}_{\chi_d}\geq \log_3(D(x))+O(1),\] where $L(s,\chi_d)$ had no Seigel zero. In light of \eqref{step}, we have at least $D(x)^{1-\epsilon}$ integers $d\leq D(x)$ for which $\chi_d$ is a real Dirichlet character and \[ \frac{\#\{ab\leq x:\:\chi(a)=\chi(b)=1\}}{\frac{1}{4}\#\{ab\leq x:\:(ab,d)=1\}} \text{ is at least as large as }1+\frac{\log\log\log{x}+O(1)}{\log\log{x}}.\] 
\end{proof}

\section{Proof Of Theorem \ref{thm:Hough2}} 
We begin by showing that \eqref{step} can be obtained, for large $x$, on the range $d\leq x$ if we assume the Generalised Riemann Hypothesis (GRH). The crucial advantage in making this assumption is that it allows us to write \[ \theta(x,\chi)\ll x^{1/2}(\log(dx))^2,\] a much stronger bound than \eqref{eq:theta}. This bound can be found in~\cite{DAV} (pg. $125$). We thus have the corresponding bound, obtained using partial summation, \be \sum_{\substack{n\leq x \\ n \text{ prime }}}\chi(n)\ll x^{1/2}(\log{x}), \text{ for }d\leq x.\ee
The procedure for arriving at \eqref{step} is then much the same as in the proof of Theorem \ref{hough1}. The only notable difference is in bounding the first term of \eqref{examine} where we use the following result of Helge Von Koch ($1901$) that assumes the Riemann Hypothesis. 
\[ \pi(x)=Li(x)+O(\sqrt{x}\log{x}).\] The next step is to note that Theorem \ref{mod1} gives us that, for large $x$, there exists $d\leq x$ such that 
\[ e^{\gamma}(\log_2{x}-2\log_3{x})\leq L(1,\chi_d),\] Also, Theorem \ref{Littlewood} ensures that \[ L(1,\chi_d)\stackrel{\text{GRH}}{\leq} (2e^{\gamma}+o(1))\log\log{x}, \:\:\:\text{ for any } d\leq x.\] In light of \eqref{curlyL}, we thus have 
\begin{align}\label{switch} \max_{d\leq x}\mathcal{L}_{\chi_d}&\sim\log\log\log{x}+O(1).
 \end{align} Now, combining \eqref{switch} with \eqref{step}, we have that \[\max_{d\leq x}\left|\frac{\#\{ab\leq x: \chi_d(a)=\chi_d(b)=1\}}{\frac{1}{4}\#\{ab\leq x: (ab,d)=1\}}\right| \sim1+\frac{\log\log\log{x}}{\log\log{x}},\] and Theorem \ref{thm:Hough2} follows.

 \section{Proof Of Theorem \ref{thm:Hough3}} 
 Following exactly the same steps as seen in the proof of Theorem \ref{hough1}, we have that for large $x$ and $d\leq\exp[C(\log{x})^{1/2}]$, where $L(s,\chi_d)$ has no Siegel zero, \begin{equation}\label{step2}  \frac{\#\{ab\leq x;\:\chi(a)=\chi(b)=1\}}{\frac{1}{4}\#\{ab\leq x;\:(ab,d)=1\}} = 1+\frac{\mathcal{L_\chi}+O(1)}{\log\log{x}}.\end{equation}
Now we use a second application of Theorem $1$ from~\cite{GS}.
\begin{theorem}[Application Of Theorem $1$ from~\cite{GS}]\label{mod2}
Fix $\epsilon>0$. For sufficiently large $D$, there are at least $D^{1-\epsilon}$ fundamental discriminants $d\leq D$ for which \[\mathcal{L}_{\chi_d}\leq -\log_3(D)+O(1),\] such that $L(s,\chi_d)$ does not have a Seigel zero.
\end{theorem}
\begin{proof}(of Theorem \ref{mod2})
The proof of this result follows is much the same way as in the proof of Theorem \ref{mod1}. Again, taking $\tau=\log_2(D)-2\log_3(D)$ in Theorem $1$ of~\cite{GS}, we have that for fixed $\epsilon$ and sufficiently large $D$, \[\Psi_D(\tau)\gg D^{-\epsilon},\] where $\Psi_D(\tau)$ is the proportion of fundamental discriminants with $d\leq D$ for which \[L(1,\chi_d)\leq \frac{\pi^2}{6e^{\gamma}\tau}.\]Now recall the result of Page which ensures that there is never more than one conductor $d\leq D$ for which $L(s,\chi_d)$ has a Siegel zero. Finally, using \eqref{curlyL} we have that for sufficiently large $D$, there are at least $D^{1-\epsilon}$ fundamental discriminants $d\leq D$ such that \[\mathcal{L}_{\chi_d}\leq -\log_3(D)+O(1).\]

 \end{proof}

\begin{proof}(of Theorem \ref{thm:Hough3})
Defining $D(x)=\exp[C(\log{x})^{1/2}]$, Theorem \ref{mod2} gives us that there are at least $D(x)^{1-\epsilon}$ integers $d\leq D(x)$ such that the corresponding Dirichlet character $\chi_d$ is real and \[ \mathcal{L}_{\chi_d}\leq -\log_3(D(x))+O(1),\] where $L(s,\chi_d)$ had no Siegel zero.  In light of \eqref{step}, we have at least $D(x)^{1-\epsilon}$ integers $d\leq D(x)$ for which $\chi_d$ is a real Dirichlet character and \[ \frac{\#\{ab\leq x:\:\chi(a)=\chi(b)=1\}}{\frac{1}{4}\#\{ab\leq x:\:(ab,d)=1\}} \text{ is at least as small as }1-\frac{\log\log\log{x}+O(1)}{\log\log{x}}.\] 
\end{proof}

 \section{Computational Aspects} 
It was Euler in 1772 who first came across what became `Euler's Polynomial' \[ x^2+x+41.\] 
What is so remarkable is that for the integers $0\leq x\leq 39$, the value of $x^2+x+41$ is prime. If we now consider the general quadratic 
\[f(x) = ax^2+bx+c,\]  then, for a given integer $x$, the value of this quadratic is divisible by a prime $p$ if and only if 
\[ax^2+bx+c\equiv 0 \mod{p},\]
which is equivalent to saying that \[(2ax+b)^2\equiv b^2-4ac \mod{p}. \] Said differently $\left(\frac{\Delta(f)}{p}\right)=0 \text{ or } 1$. This suggests that if a given quadratic produces a high density of primes for integers $x\leq m$, then the character corresponding to $\chi(n)=\left(\frac{\Delta(f)}{n}\right)$ might produce a high density of $-1$ values for primes $n\leq m$. 
We now turn to an article by Jacobson and Williams ~\cite{JW}. Their work is based upon a conjecture of Hardy and Littlewood, namely `Conjecture F'. This implies that, for a polynomial of the form $f(x)=x^2+x+A$, $A\in \mathbb{Z}$ with discriminant $\Delta$, the asymptotic density of prime values of $f$ is related to a quantity $C(\Delta)$. They also suggest that the larger the value of $C(\Delta)$, the higher the asymptotic density of primes for \emph{any} polynomial of discriminant $\Delta$. We thus restrict to polynomials of the form $f(x)=x^2+x+A$. If we also denote by $P_A(n)$, the number of primes produced by $f_A(x)$ for $0\leq x\leq n$ then Conjecture $F$ can be written as follows.
\begin{conjecture}[Conjecture F (simplified), Hardy and Littlewood, $1923$]
\[ P_A(n)\sim C(\Delta)L_A(n),\] where \[L_A(n)=2\int_{0}^{n}\frac{dx}{\log{f_A(x)}}\,, \] and \[C(\Delta)=\prod_{p\geq 3}{1-\frac{\left(\frac{\Delta}{p}\right)}{p-1}}.\]
\end{conjecture}
Jacobson and Williams predict that the polynomial $x^2+x+A$ `has the highest asymptotic density of prime values for any polynomial of this type currently know', where $A$ is given by 
\begin{equation}\label{long1} -33251810980696878103150085257129508857312847751498190349983874538507313.\end{equation}
We would thus expect the character with conductor $d$ given by $\chi(n)=(d/n)$ where $d=1-4A$  to yield a low value of $\mathcal{L}_{\chi_d}$. Firstly, we note that $d\equiv 1\mod{4}$ is square-free and thus $(d/n)$ defines a real primitive character. A simple piece of SageMath yields 
\[\mathcal{L}_{(d/.)}\approx -2.1108.\]
Furthermore, the data confirms this bias. 
Defining \[r_d(x):=\frac{\#\{pq\leq x:\:\chi_d(p)=\chi_d(q)=\eta\}}{\frac{1}{4}\#\{pq\leq x: \: (pq,d)=1\}},\] where $d$ is the conductor of the character $\chi_d$, we have that, on taking $d=1-4A$ where $A$ is given by \eqref{long1} and $\eta=-1$,
\[ r_d(10^3)\approx 3.847,\:\:r_d(10^4)\approx 2.974,\:\: r_d(10^5)\approx 2.394,\:\: r_d(10^6)\approx 2.067. \]
This is the largest bias calculated for such primes races amongst products of two primes. In light of \eqref{curlyL}, it is no surprise that this conductor also gives rise to what we believe to be the lowest known value of $L(1,\chi_d)$, where $\chi_d$ is a general Dirichlet character of modulus $d$. \[L(1,\chi_d)=0.144.\]

\end{document}